\documentclass[a4paper, 12pt, reqno]{amsart}
\usepackage[utf8]{inputenc}

\usepackage{amsmath}
\usepackage{amsfonts}
\usepackage{amssymb}
\usepackage{lipsum}
\usepackage{physics}
\usepackage{esint}

\usepackage{amsthm}
\usepackage{enumerate}
\usepackage{amsrefs}

\usepackage{graphicx}
\usepackage{tikz}

\usepackage{threeparttable}
\usepackage{float} 

\usepackage[titletoc]{appendix}
\usepackage[english]{babel}
\usepackage[utf8]{inputenc}
\usepackage[none]{hyphenat}

\usepackage{geometry}
\usepackage{hyperref}
\hypersetup{
    colorlinks=true,
    linkcolor=blue,
    citecolor=blue,
    filecolor=magenta,      
    urlcolor=cyan,
}

\newcounter{foo}

\theoremstyle{plain}

\newtheorem{thm}[foo]{Theorem}
\newtheorem{prop}[foo]{Proposition}
\newtheorem{lem}[foo]{Lemma}

\theoremstyle{definition}

\theoremstyle{remark}
\newtheorem{rem}[foo]{Remark}

\swapnumbers

\newcommand{\Ric}{\text{Ric}}
\newcommand{\Rm}{\text{Rm}}

\newcommand{\RR}{\mathbb{R}}

\usepackage{hyperref}
\usepackage{cleveref}

\numberwithin{foo}{section}
\numberwithin{equation}{section}

\title{Singularity Type for Long Time Chern-Ricci Flows}
\author{Hosea Wondo}

\begin{document}
\maketitle

\begin{abstract}
    We extend some results known for the K\"ahler-Ricci flow to the Chern-Ricci flow regarding the independence of singularity types for long time solutions. Specifically, we show that if a solution to the Chern-Ricci flow exists with uniformly bounded torsion and curvature, then any other solution starting from a initial metric of the same $\partial \bar{\partial}$ class will also exhibit uniform bounds on torsion and curvature. 
\end{abstract}

\section{Introduction}

The Chern-Ricci flow is one of the natural extensions of the K\"ahler-Ricci flow in Hermitian geometry. Rather than evolving the  K\"ahler metric by the Ricci curvature of the Levi-Civita connection, one instead evolves the Hermitian metric by the Chern-Ricci form of the Chern connection. The normalised variant of this geometric flow is given by 

\begin{equation}\label{1CRF}
   \left\lbrace \begin{aligned}
        \pdv{\omega}{t} &= - \Ric(\omega) - \omega, \\
        \omega(0)&= \omega_0.
    \end{aligned}\right.
\end{equation}

The Chern-Ricci flow was first studied by Gill in \cite{gill2010convergence} on manifolds with vanishing first Bott-Chern class. A general study of the flow was given by Tosatti and Weinkove in \cite{tosatti2015evolution} where the authors conjecture a non-K\"ahler variant of the analytic minimal model program for complex surfaces. Naturally, one hopes such a flow could produce classification results for Hermitian manifolds, particularly for complex VII surfaces \cite{tosatti2013chern}. Furthermore, flows in the non-K\"ahler setting are motivated by a growing interest in developing a theory for non-K\"ahler Hermitian manifolds. 

As with the K\"ahler-Ricci flow, the Chern-Ricci flow's maximal time of existence is characterised by a particular cohomology class. In \cite{tosatti2015evolution} it was shown that solutions to the equation \eqref{1CRF} exists up to time 
\begin{equation}\label{1cohomologyT}
    T=\sup \left\{t>0 \mid \exists \psi \in C^{\infty}(M) \text { with } e^{-t}\omega_0+(1-e^{-t}) \operatorname{Ric}\left(\omega_0\right)+\sqrt{-1} \partial \bar{\partial} \psi>0\right\}
\end{equation}
This is the analgous result to the maximal time of existence result for the K\"ahler setting in \cite{tian2006kahler}. Since we are not necissarily in the K\"ahler setting, the K\"ahler forms do not generate a cohomology class. Instead, \eqref{1cohomologyT} is considered to be the path in the Bott-Chern cohomology. This cohomology is defined by 
\begin{equation}
    H_{B C, \mathbb{R}}^{1,1}:=\frac{\left\{\operatorname{Ker} d: \Lambda_{\mathbb{R}}^{1,1} \rightarrow \Lambda_{\mathbb{R}}^{2,2}\right\}}{\left\{\sqrt{-1} \partial \bar{\partial} f \mid f \in C^{\infty}(X)\right\}} .
\end{equation}

Locall, the Chern-Ricci form is given by 
\begin{equation}
    \Ric(\omega) =  -\sqrt{-1} \partial \bar{\partial} \log \det g_{i \bar{j}} 
\end{equation}
but does not conicide, unless $\omega$ is K\"ahler, with the Levi-Civita Ricci curvature nor the second Ricci curvature of the Chern connection.

In \cite{tosatti2015collapsing} it was shown that $T = \infty $ is characterized by the canonical line bundle of $K_X$ being numerically effective. That is, for all $\epsilon >0$, there exists $\psi_\varepsilon \in C^\infty(X)$ such that 
$$-\operatorname{Ric}\left(\omega_0\right)+\sqrt{-1} \partial \bar{\partial} \psi_{\varepsilon}>-\varepsilon \omega_0 . $$
As with the K\"ahler case, the nef condition is independent of time $\omega_0$ and the limiting class is also independent on the starting metric. With this in mind, it is natural to study which properties of the flow hold for different starting metrics. For the K\"ahler-Ricci flow, there is strong evidence that curvature boundedness is preserved for different starting metrics \cite{tosatti2015infinite} and it was conjectured by Tosatti that the curvature type is independent of the inital metric \cite{tosatti2018kawa}. This turns out to be true in the K\"ahler case \cites{zhang2020infinite,wondo2023independence}.  

As the Chern-Ricci flow is a natural extension of the K\"ahler-Ricci flow, we can ask whether the same results holds with the prsenece of torsion. In this paper, we prove a partial result to this, and show the following type-independence result for curvature and torsion. 

\begin{thm}\label{T1}
Let $\widetilde{\omega}(t)$ be a solution to the normalised Chern-Ricci flow \eqref{1CRF} on a compact Hermitian manifold $(X,\omega_0)$ with numerically effective $K_X$. Suppose that 
\begin{equation}\label{T1Assumption}
    ||\widetilde{T}(t)||_{\widetilde{\omega}(t)}^2 + ||\text{{\normalfont Rm}}(\widetilde{\omega}(t))||^2_{\widetilde{\omega}(t)} \leqslant C_0. 
\end{equation}
Then for any other solution to the Chern-Ricci flow, $\omega(t)$ starting from a metric $\omega_0$ satisfying $\omega_0 = \widetilde{\omega}_0  + \sqrt{-1} \partial \bar{\partial} \varphi_0$ for some $\phi_0 \in C^\infty(X,\RR)$, there exists a constant $C \geqslant 0$ such that
\begin{equation}\label{ThmRmbdd}
    ||T(t)||_{\omega(t)}^2 + ||\text{{\normalfont Rm}}  (\omega(t))||^2_{\omega(t)}\leqslant C
\end{equation}
for all $t \geqslant 0$. 
\end{thm}

When $\partial \omega_0 = \partial \widetilde{\omega}_0 =  0$, that is both metrics on $X$ is K\"ahler, then the flow equation \eqref{1CRF} is simply the K\"ahler-Ricci flow and we recover the type-independence type for the K\"ahler case \cites{wondo2023independence,zhang2020infinite}. Long time solutions to the K\"ahler-Ricci flow have been extensively studied, particularly in the context of the final step in the analytic minimal model program \cite{song2006thekahler} (for a recent survey on immortal K\"ahler-Ricci flows, see \cite{tosatti2024immortal}). 

We also mention an alternative flow due to Streets and Tian \cites{tian2011hermitian,streets2010parabolic,streets2020pluriclosed}. The pluriclosed flow is a subset of a larger class of flows called Hermitian flows. The pluriclosed flow evolves the Hermitian metric by 
\begin{equation}\label{1pluriclosed}
    \frac{\partial}{\partial t} \omega=-S+Q
\end{equation}
where $S$ is the second Ricci curvature of the Chern conection and $Q$ is a quadratic term involving torsion
$$
    Q_{i \bar{j}}=g^{\bar{l} k} g^{\bar{q} p} T_{i k \bar{p} \overline{ }} T_{\bar{j} \bar{l} q}.
$$
The flow \eqref{1pluriclosed} is called the Pluriclosed flow. It manifests as a gradinet flow of the lowest eigenvalue of a certain Schr \"odinger operator \cite{streets2013regularity} and has links to Hitchen's generallised geoemtry \cites{hitchin2003generalized,streets2012generalized,streets2016pluriclosed}. Unlike the Chern-Ricci flow, the Pluriclosed flow does not exhibit a reduction to a scalar Monge-Ampere equation.

This paper is orgainsed as follows. In Section 2, we reduce the two flows begining at different metrics to a Monge-Ampere equation. We derive $L^\infty(X)$ estiamtes fo the difference in the potential function obtained from this reduction. In Section 3, we derive metric equivalence between the solutions $\omega(t)$ and $\widetilde{\omega}$. Here, we make the assumption that the inital metrics reside in the same $\partial \bar{\partial}$ in order to deal with the torsion terms we need to contend with. Finally, in Section 4, we derive higher order estimtes for the flow $\omega(t)$. The metric equivalence allows us to trasnfer torsion and curvature bounds from $\widetilde{\omega}(t)$ to the $\omega(t)$. \\

\textbf{Acknowledgments.} The author would like to thank his advisor, Zhou Zhang for countless interesting discussions and Valentino Tosatti for his suggestions and interest in this work.  

\section{Monge-Ampere Equations}

We begin by relating two solutions to the flow through a Monge-Ampere equation. Define a function $v:X \times [0,\infty) \rightarrow \RR$ satisfying  
\begin{equation}\label{2v}
    v(p,t) := e^{-t} \int^t_0 e^s \log\left( \frac{\omega^n}{\widetilde{\omega}^n}\right).
\end{equation}
We suppress the $p$ and denote $v(p,t)$ as $v(t)$ for notational convenience. This function solves the following parabolic Monge-Ampere equation given by 

\begin{equation}\label{2MA}
    \left\lbrace \begin{aligned}
        \pdv{v}{t} &= \log \left( \frac{\omega^n}{\widetilde{\omega}^n} \right) - v(t) \\
        v(0) &= 0. \\
    \end{aligned}\right.
\end{equation}
Then 
\begin{equation}
\Phi := \omega(t)-\widetilde{\omega}(t)-e^{-t}\left(\omega_0-\widetilde{\omega}_0\right)-\left(1-e^{-t}\right)\left(\Ric(\omega_0)-\Ric(\widetilde{\omega}_0)\right) - \sqrt{-1} \partial \bar{\partial} v.
\end{equation}

Taking a time derivative of $\Phi$ and using the flow equation \eqref{1CRF} and the Monge-Ampere equation \eqref{2MA}, we obtain 
\begin{equation*}
    \begin{aligned}
        \pdv{\Phi}{t} &= \frac{\partial \omega}{\partial t}-\frac{\partial \tilde{\omega}}{\partial t}+e^{-t}\left(\omega_0-\widetilde{\omega}_0\right)-e^{-t}\left(\operatorname{Ric}\left(\omega_0\right)-\Ric(\widetilde{\omega}_0 \right)) -\sqrt{1} \partial \bar{\partial}\left(\frac{\partial v}{\partial t}\right) \\
        &= -\omega(t)+\widetilde{\omega}(t)+e^{-t}\left(\omega_0 -\widetilde{\omega}_0\right)-e^{-t}\left(\operatorname{Ric}\left(\omega_0\right)-\Ric(\widetilde{\omega}_0 \right)) + \sqrt{-1}\partial \bar{\partial} v \\
        &= - \Phi + \sqrt{-1} \partial \bar{\partial} \log \left( \frac{\omega_0^n}{\widetilde{\omega}_0^n} \right).
\end{aligned}
\end{equation*}
Therefore, the above implies 
$$
    \Phi(t) = \sqrt{-1} \partial \bar{\partial} \log \left( \frac{\omega_0^n}{\widetilde{\omega}_0^n} \right) + Fe^{-t}.
$$
for some time independent $(1,1)$ form $F$. To recover $F$, we note that the initial conditions enforce $\Phi_0 = 0$, which implies $F= - \sqrt{-1} \partial \bar{\partial} \log \left( \omega_0^n / \widetilde{\omega}_0^n \right)$. Thus, we have 
$$
    \Phi(t) = (1-e^{-t}) \sqrt{-1} \partial \bar{\partial} \log \left( \frac{\omega_0^n}{\widetilde{\omega}_0^n} \right)
$$
which implies 
\begin{equation}\label{2Cohomology}
    \omega(t) = \widetilde{\omega}(t) + e^{-t}\left(\omega_0-\widetilde{\omega}_0\right) +  \sqrt{-1} \partial \bar{\partial} v.
\end{equation}

\begin{rem}
    In the K\"ahler case, the equation \eqref{2Cohomology} follows immediately from the cohomology classes induced by the K\"ahler form. 
\end{rem}

Consider two Chern-Ricci flows, starting from $\omega_0$ and $\widetilde{\omega}_0$. We first show that the potential that arises out of such scaling is equivalent. That is, if $\omega_0 = \lambda \omega_0$ for some $\lambda>0$ and $\omega(t), \widetilde{\omega}(t)$ solve \eqref{1CRF}, then the potential $v$ defined by \eqref{2v} are uniformly bounded in time. Using this result, we can show the general case.

\begin{prop}\label{2prop1}
    Let $\omega(t)$ and $\widetilde{\omega}(t)$ be two solutions to \eqref{1CRF} with initial conditions related by 
    \begin{equation}\label{2prop1A}
        \omega_0 = \lambda_0 \widetilde{\omega}_0
    \end{equation}
    for some $\lambda_0>0$. Suppose that 
    \begin{equation}\label{2prop1A2}
       -C \widetilde{\omega}  \leqslant \Ric(\widetilde{\omega}) \leqslant C \widetilde{\omega}.
    \end{equation}
    Then $v$ defined in \eqref{2v} satisfies 
    \begin{equation}
        |v|_{L^\infty(X)} + \left|\pdv{v}{t} \right|_{L^\infty(X)} \leqslant C
    \end{equation}
    for some $C>0 $ independent of time. 
\end{prop}

\begin{proof}
    For some $\lambda_0>0$, consider the scaled initial metric 
    $$ \omega_0 = \lambda_0 \widetilde{\omega}_0.$$
    If $\widetilde{\omega}(t)$ solves the Chern-Ricci flow, then $\omega(t) = \lambda(t) \widetilde{\omega}(\tau)$ solves the Chern-Ricci flow where 
    \begin{equation}
        \left\lbrace \begin{aligned}
            \lambda(t) &= 1 + (\lambda_0 - 1)e^{-t} \\
            \tau(t) &= t + \ln \left( \frac{\lambda(t)}{\lambda_0} \right).
        \end{aligned}\right.
    \end{equation}
    Note that $\tau = t + o(t)$ and $\tau \rightarrow \infty$ as $t \rightarrow \infty$. The geometries of these two flows are related as follows:
    \begin{equation}
        \omega(t) = (1+(\lambda_0-1)e^{-t})\widetilde{\omega}(\tau) 
    \end{equation}
    \begin{equation}
        T(t) = \widetilde{T}(\tau)
    \end{equation}
    \begin{equation}
        R(t) = \frac{\widetilde{R}(\tau)}{1+(\lambda_0-1)e^{-t}} 
    \end{equation}
    \begin{equation}
        \Ric(\omega(t)) = \Ric(\widetilde{\omega}(\tau))
    \end{equation}
    \begin{equation}
        \Rm(\omega(t)) = (1+(\lambda_0-1)e^{-t})\Rm(\widetilde{\omega}(\tau)).
    \end{equation}
    In particular, a uniform bound on scalar curvature of $\widetilde{\omega}(\tau)$ implies a uniform bound on scalar curvature of $\omega(t)$. The potential defined as in \eqref{2v} is 
    \begin{equation}\label{2vscaled}
    \begin{aligned}
        v &= e^{-t} \int_0^t e^s \log \left( \frac{\omega(t)^n}{\widetilde{\omega}(t)^n} \right) + e^{-t} \log \lambda_0^n \\
        &= e^{-t} \int_0^t e^s \log \lambda(s) ds + e^{-t} \log \lambda_0^n + e^{-t} \int_0^t e^s \log \left( \frac{\widetilde{\omega}(\tau(s))^n}{\widetilde{\omega}(s)^n} \right)ds.
    \end{aligned}
    \end{equation}
    The first two terms in the last line above are bounded. To estimate the last term, we apply the Ricci curvature bound \eqref{2prop1A2} on the Chern-Ricci flow to obtain 
    $$
        \pdv{}{\tau} \left( e^{-(c_0-1)\tau} \widetilde{\omega}(\tau)\right) \leqslant 0
    $$
    and 
    $$
        \pdv{}{\tau} \left( e^{(c_0+1)\tau} \widetilde{\omega}(\tau)\right) \geqslant 0.
    $$
    Taking cases, if $\lambda_0>1$, then $\tau>t$ for all $t \geqslant 0$. On the other hand, if $0<\lambda_0<1$, then $\tau <t$ for all $t \geqslant 0$. Combining this with the two estimates above, we obtain distortion estimates on $\widetilde{\omega}$ which yields a bound on the volume ratio $\widetilde{\omega}(\tau)^n/\widetilde{\omega}(t)^n$.
\end{proof}

\begin{prop}\label{2prop2}
    The conclusion in Proposition \ref{2prop1} holds without assumption \eqref{2prop1A}.
\end{prop}

\begin{proof}
    Let $\omega^+_0 := \lambda_0^+ \widetilde{\omega}_0 $ and $\omega^-_0 := \lambda_0^- \widetilde{\omega}_0 $ for $\lambda_0^+ >1$ and $0 < \lambda_0^- < 1$. Using \eqref{2v}, we define 
    \begin{equation}
    u := e^{-t} \int^t_0 e^s \log\left( \frac{\omega^n}{(\omega^{-})^n}\right),
    \end{equation}
    and similarly, 
        \begin{equation}
    v := e^{-t} \int^t_0 e^s \log\left( \frac{(\omega^+)^n}{\omega^n}\right),
    \end{equation}
    and 
        \begin{equation}
    \psi := e^{-t} \int^t_0 e^s \log\left( \frac{(\omega^{-})^n}{(\omega^{+})^n}\right).
    \end{equation}
Note that these function are related by 
\begin{equation}\label{2uvpsi}
    \psi = u + v.
\end{equation}
Using a maximum principle argument, we can derive a lower bound for $u$ and $v$. We choose $\lambda_0$ small enough so that $\omega > \omega_0^-$. Then \eqref{2Cohomology} implies that 
\begin{equation}\label{2est}
        \omega(t) \geqslant \omega^-(t) + \sqrt{-1} \partial \bar{\partial} u.
\end{equation}
For each $t>0$, at a minimum of $u$  denoted as $u|_{\text{min}}$, we have from \eqref{2est} the estimate 
\begin{equation}
    \begin{aligned}
        \pdv{u}{t} \geqslant \log(1)  - u|_{\text{min}} = - u_{\text{min}}.
    \end{aligned}
\end{equation}
The above inequality implies
$$
    \pdv{}{t} \left( e^t u|_{\text{min}} \right) \geqslant 0.
$$
That is, the quantity in the brackets above is increasing by $t$. Therefore, it is bounded from below by its value at $t=0$. Putting all this together yields the estimate 
\begin{equation}
   e^t u \geqslant e^t u_{\text{min}} \geqslant 0.
\end{equation}
Thus, we have  
$u \geqslant 0$ for all $t >0$. \\

The same argument holds for $v$ given we choose $\lambda_0^+$ large enough so that $\omega^+_0 > \omega_0$. Since $\omega_0^+ = (\lambda_0^+/\lambda_0^-)\omega_0^-$, Proposition \ref{2prop1} implies $|\psi|_{L^\infty(X)} \leqslant C$. Combining the above allows us to conclude 
$$
0 \leqslant u = \psi - v \leqslant C
$$
for some $C>0$. Once again, by \eqref{2uvpsi}, we have $|v|_{L^\infty(X)} \leqslant C$ for some $C>0$.\\

We now estimate the time derivatives. From \eqref{2Cohomology}, we have 
\begin{equation}\label{2heatu}
    \left(\pdv{}{t} - \Delta_{\omega(t)}\right) u = \log \left( \frac{\omega^n}{(\omega^-)^n} \right) - u + n + \tr_{\omega} \omega^- + e^{-t}\tr_{\omega}(\omega_0 - \omega^-_0).
    \end{equation}
 On the other hand, taking a second order time derivative of $u$ which satisfies \eqref{2MA}, we have 
 \begin{equation}\label{2heatu'}
      \left( \pdv{}{t}  - \Delta_{\omega}\right) \left( \pdv{u}{t} \right) = -\pdv{u}{t} + R^- - \tr_{\omega}(\Ric(\omega^-)) - \tr_{\omega} \omega^- - e^{-t} \tr_{\omega}(\omega_0 - \omega^-_0) + n.
 \end{equation}
Now let 
\begin{equation}
    Q:= \pdv{u}{t} -Au
\end{equation}
for some large $A>0$, that will be set later. Using \eqref{2heatu} and \eqref{2heatu'}, the evolution of $Q$ is given by 
$$
\begin{aligned}
    \left(\frac{\partial}{\partial t}-\Delta_\omega\right) Q & = - (A+1) \pdv{u}{t} + R^{-}-\operatorname{tr}_\omega \operatorname{Ric}\left(\omega^{-}\right)-(A+1) \operatorname{tr}_\omega \omega^{-} \\
    &-(A+1) e^{-t} \operatorname{tr}_\omega (\omega_0-\omega_0^-)-(A+1) \frac{\partial u}{\partial t}-(A+1) n \\
    & \leqslant-\left(A+1-C_0\right) \operatorname{tr}_\omega \omega^{-} -(A+1) \frac{\partial u}{\partial t} + C,
\end{aligned}
$$
where the inequality in the last line is obtained by applying \eqref{T1Assumption} to bound the curvature terms in the first line and noting that $\omega_0 > \omega_0^-$ to ignore the first term in the second line. Here, $C>0$ is a constant independent of time.  We now set $A= 2+C_0$ to force a negative sign on the first term. This term can thus be ignored in the estimate. Applying the maximum principle, we have $Q \leqslant C$ for some $C>0$ independent of time. Along with proposition \eqref{2prop2}, we obtain 
\begin{equation}
    \pdv{u}{t} \leqslant C.
\end{equation}
For the lower bound, we consider 
\begin{equation}
    Q:= \pdv{u}{t}+Au.  
\end{equation}
Similar to the above, 
$$
\begin{aligned}
\left(\frac{\partial}{\partial t}-\Delta_\omega\right) Q & =R^{-}(t)-\operatorname{tr}_\omega \operatorname{Ric}\left(\omega^{-}\right)+(A-1) \operatorname{tr}_\omega {\omega^-} \\
& +(A-1) e^{-t} \operatorname{tr}_\omega (\omega_0 -\omega_0^-)+(A-1) \frac{\partial u}{\partial t}+(A-1)n - (A+1)e^{-t} \log \left( \frac{\omega_0^n}{(\omega_0^-)^n} \right) \\
& \geqslant\left(C_0+1-A\right) n+(A-1) \frac{\partial u}{\partial t}+\left(A-C_0-1\right) \operatorname{tr}_\omega \omega^{-}.
\end{aligned}
$$
Choose $A = C_0+2$ and use \eqref{2MA} to obtain 
$$
\begin{aligned}
\left(\frac{\partial}{\partial t}-\Delta_\omega\right) Q & \geqslant-C+C \frac{\partial u}{\partial t}+\operatorname{tr}_\omega \omega^{-} \\
& \geqslant-C+C \frac{\partial u}{\partial t}+n\left(\frac{\left(\omega^{-}\right)^n}{\omega^n}\right)^{\frac{1}{n}} \\
& =-C+C \frac{\partial u}{\partial t}+n e^{-\frac{1}{n}\left(\frac{\partial u}{\partial t}+u\right)}.
\end{aligned}
$$
Applying a maximum principle argument gives us the lower bound 
$$ \pdv{u}{t} \geqslant -C$$
for some $C>0$ independent of time. The gives us all the bounds claimed in the Proposition. 
\end{proof}

\section{Trace Estimates}

We adapt the computations from \cite{sherman2013local} to our case. We first need the evolution of $\tr_{\omega} \widetilde{\omega}$ under the heat operator, where $\omega$ and $\widetilde{\omega}$ are two solutions to the Chern-Ricci flow. 

From this point onwards, we will assume that 
\begin{equation}\label{A!'}
    \omega_0 = \widetilde{\omega}_0 + \sqrt{-1} \partial \bar{\partial} \varphi_0
\end{equation}
for some $\varphi_0 \in C^\infty(X,\RR)$. Ideally, we would like to remove this assumption on the initial metric. Regardless, considering metrics in the same $\partial \bar{\partial}$ seem to be a natural assumption \cites{angella2023leafwise, zheng2017chern}. 

Next, we record some identities involving torsion. Since $\omega$ may not be K\"ahler, the Bianchi identities do not always hold. Instead, the symmetries of the curvature tensor are given by 
\begin{equation}\label{3RSym}
\begin{aligned}
& R_{i \bar{j} k \bar{l}}-R_{k \bar{j} \bar{i} \bar{l}}=-\nabla_{\bar{j}} T_{i k \bar{l}} \\
& R_{i \bar{j} k \bar{l}}-R_{i \bar{l} k \bar{j}}=-\nabla_i T_{\bar{j} \bar{l} k} \\
& R_{i \bar{j} k \bar{l}}-R_{k \bar{l} i \bar{j}}=-\nabla_{\bar{j}} T_{i k \bar{l}}-\nabla_k T_{\bar{j} \bar{l} i}=-\nabla_i T_{\bar{j} \bar{l} k}-\nabla_{\bar{l}} T_{i k \bar{j}} \\
& \nabla_p R_{i \bar{j} k \bar{l}}-\nabla_i R_{p \bar{j} k \bar{l}}=-T_{p i}{ }^r R_{r \bar{j} k \bar{l}} \\
& \nabla_{\bar{q}} R_{i \bar{j} k \bar{l}}-\nabla_{\bar{j}} R_{i \bar{q} k \bar{l}}=-T_{\bar{q} \bar{j}}{ }^{\bar{s}} R_{i \bar{s} k \bar{l}} .
\end{aligned}
\end{equation}

From \eqref{2Cohomology} and \eqref{A!'}, we have
\begin{equation}
    \partial \omega = \partial \widetilde{\omega}.
\end{equation}

From metric equivalence, 
\begin{equation}
    \nabla_k g_{i \bar{j}} = \partial_k g_{i \bar{j}} - \Gamma_{k i}^p g_{p \bar{j}}  = 0.
\end{equation}
Skew symmetrising gives us the torsion tensor since
\begin{equation}
    T_{ki \bar{j}} = T_{ki}^p g_{p \bar{j}} = (\Gamma_{ki}^p - \Gamma_{ik}^p) g_{p \bar{j}} =  \partial_k g_{i \bar{j}} - \partial_i g_{k \bar{j}} = (\partial \omega)_{k i \bar{j}}
\end{equation}
and thus  
\begin{equation}\label{3TorsionCnst}
    T_{ki \bar{j}} = \widetilde{T}_{ki \bar{j}}. 
\end{equation}

\begin{prop}
    Assuming \eqref{A!'}, then 
    \begin{equation}
        \left( \pdv{}{t} - \Delta_{\omega} \right) \log \tr_{\widetilde{\omega}}\omega \leqslant \frac{2}{(\operatorname{tr}_{\widetilde{\omega}} \omega)^2} \operatorname{Re}\left( g^{\bar{q} k}\widetilde{T}_{i k}^i \partial_{\bar{q}} \operatorname{tr}_{\widetilde{g}} g \right) + C_0 \left( \tr_{\omega} \widetilde{\omega} + \frac{\tr_{\omega} \widetilde{\omega}}{\tr_{\widetilde{\omega}} \omega} + 1. \right) 
    \end{equation}
    where $C_0>0$ is a bound depending on the curvature of $\widetilde{\omega}$.
\end{prop}

\begin{proof}
From the calculations in \cite{sherman2013local} we have 
 \begin{equation}\label{3tr_xx}
\begin{aligned}
\Delta \operatorname{tr}_{\widetilde{\omega}} \omega &= g^{\bar{j} i} \widetilde{g}^{\bar{l} k} \left( \widetilde{\nabla}_{\bar{l}} \widetilde{\nabla}_k g_{i \bar{j}}+  \widetilde{\nabla}_{\bar{\ell}}\widetilde{T}_{i k \bar{j}} +\widetilde{\nabla}_i \widetilde{T}_{\bar{j}k \bar{l}} \right. \\
&  -\left(\widetilde{\nabla}_i \overline{\widetilde{T}_{j \ell}^q}-\widetilde{R}_{i \bar{\ell} p \bar{j}} \widetilde{g}^{\bar{q} p}\right) g_{k \bar{q}} -\left(\widetilde{\nabla}_{\bar{\ell}} \widetilde{T}_{i k}^p+\widetilde{R}_{i \bar{\ell} k \bar{q}} \widetilde{g}^{\bar{q} p}\right) g_{p \bar{j}} \\
& \left.-\overline{\widetilde{T}_{j \ell}^q} \widetilde{\nabla}_k g_{i \bar{q}} -\widetilde{T}_{i k}^p \widetilde{\nabla}_{\bar{\ell}} g_{p \bar{j}} - \overline{\widetilde{T}_{j \ell}^q}\widetilde{T}_{i k}^p\widetilde{g}_{p \bar{q}}+\widetilde{T}_{i k}^p \overline{\widetilde{T}_{j \ell}^q} g_{p \bar{q}}\right).
\end{aligned}
\end{equation}

The time derivative for the normalised flow is given by 
 \begin{equation}
\frac{\partial}{\partial t} \operatorname{tr}_{\widetilde{\omega}} \omega = \widetilde{R}^{\bar{j}i} g_{i \bar{j}} + \widetilde{g}^{\bar{\imath} k} \partial_k \partial_{\bar{\ell}} \log \operatorname{det}(g) = \widetilde{R}^{\bar{l}k} g_{i \bar{j}}+   g^{\bar{j} i} \widetilde{g}^{\bar{\ell} k} \partial_k \partial_{\bar{\ell}} g_{i \bar{j}}-g^{\bar{j} p} g^{\bar{q} i} \widetilde{g}^{\bar{\ell} k} \partial_k g_{\bar{i} j} \partial_{\bar{\ell}} g_{p \bar{q}}.
\end{equation}
Rewriting the above in terms of covariant derivatives gives us 
\begin{equation}\label{3tr_t}
    \frac{\partial}{\partial t} \operatorname{tr}_{\widetilde{\omega}} \omega = \widetilde{R}^{\bar{j}i} g_{i \bar{j}}+ g^{\bar{j} i} \widetilde{g}^{\bar{\ell} k} \widetilde{\nabla}_{\bar{\ell}} \widetilde{\nabla}_k g_{i \bar{j}}-g^{\bar{j} p} g^{\bar{q} i} \widetilde{g}^{\bar{\ell} k} \widetilde{\nabla}_k g_{i \bar{j}} \widetilde{\nabla}_{\bar{\ell}} g_{p \bar{q}}-\widetilde{g}^{\bar{\ell} k} \widetilde{g}^{\bar{j} i} \widetilde{R}_{k \bar{\ell} i \bar{j}}.
\end{equation}

Combining \eqref{3tr_xx} and \eqref{3tr_t}, we obtain 
\begin{equation}\label{3logevol'}
\begin{aligned}
        \left( \pdv{}{t} - \Delta_{\omega} \right) \log \tr_{\widetilde{\omega}} \omega &= \frac{1}{\tr_{\widetilde{\omega}} \omega} \left(   -g^{\bar{j} p} g^{\bar{q} i} \widetilde{g}^{\bar{\ell} k} \widetilde{\nabla}_k g_{i \bar{j}} \widetilde{\nabla}_{\bar{\ell}} g_{p \bar{q}} -\frac{1}{\operatorname{tr}_{\widetilde{g}} g} g^{\bar{\ell} k} \widetilde{\nabla}_k \operatorname{tr}_{\widetilde{g}} g \widetilde{\nabla}_{\bar{\ell}} \operatorname{tr}_{\widetilde{g}} g\right.\\ 
        & -2 \operatorname{Re}\left(g^{\bar{j} i} \widetilde{g}^{\bar{\ell} k} \widetilde{T}_{k i}^p \widetilde{\nabla}_{\bar{\ell}} g_{p \bar{j}}\right)-g^{\bar{j}} \widetilde{g}^{\bar{\ell} k} \widetilde{T}_{i k}^p \overline{\widetilde{T}_{j \ell}^q} g_{p \bar{q}} \\
        & + g^{\bar{j} i} \widetilde{g}^{\bar{\ell} k}\left(\widetilde{\nabla}_i \overline{\widetilde{T}_{j \ell}^q}-\widetilde{R}_{i \bar{\ell} p j} \widetilde{g}^{\bar{q} p}\right) g_{k \bar{q}} + \widetilde{g}^{\bar{l} k}\left(\widetilde{\nabla}_{\bar{\ell}} \widetilde{T}_{i k}^i+\widetilde{R}_{i \bar{l} k \bar{q}}  \widetilde{g}^{\bar{q} i}-\widetilde{R}_{k \bar{l} i \bar{q}} \widetilde{g}^{\bar{q} i}\right) \\
        & - \left. g^{\bar{j} i} \widetilde{g}^{\bar{l} k} \widetilde{\nabla}_{\bar{\ell}}\widetilde{T}_{i k \bar{j}} - g^{\bar{j} i} \widetilde{g}^{\bar{l} k} \widetilde{\nabla}_i \widetilde{T}_{\bar{j}k \bar{l}} + \widetilde{R}^{\bar{j}i} g_{i \bar{j}} \right).
\end{aligned}
\end{equation}
From the identity, we have 
$$
\widetilde{R}_{i \bar{\ell} k \bar{q}}-\widetilde{R}_{k \bar{\ell} i \bar{q}}=-\widetilde{g}_{j \bar{q}} \partial_{\bar{\ell}} \widetilde{T}_{i k}^j=-\widetilde{g}_{j \bar{q}} \widetilde{\nabla}_{\bar{\ell}} \widetilde{T}_{i k}^j
$$
which implies 
$$
\widetilde{g}^{\bar{e} k} \widetilde{g}^{\bar{q} i}\left(\widetilde{R}_{\bar{\ell} k \bar{q}}-\widetilde{R}_{k \bar{\ell} i \bar{q}}\right)=-\widetilde{g}^{\bar{e} k} \widetilde{g}^{\bar{q} i} \widetilde{g}_{j \bar{q}} \widetilde{\nabla}_{\bar{\ell}} \widetilde{T}_{i k}^j=-\widetilde{g}^{\bar{\ell} k} \widetilde{\nabla}_{\bar{\ell}} \widetilde{T}_{i k}^i.
$$
Rearranging gives us 
\begin{equation}\label{3zero}
\widetilde{g}^{\bar{\imath} k}\left(\widetilde{\nabla}_{\bar{\ell}} \widetilde{T}_{i k}^i+\widetilde{R}_{i \bar{\ell} k \bar{q}} \widetilde{g}^{\bar{q} i}-\widetilde{R}_{k \bar{\ell} i \bar{q}} \widetilde{g}^{\bar{q} i}\right)=0.
\end{equation}
Applying \eqref{3logevol'} and following Cherrier's argument on the remaining terms, we can find some $K \geqslant 0$ such that  
\begin{equation}\label{3logevol'}
\begin{aligned}
       \left( \pdv{}{t}  - \Delta_{\omega} \right) & \log \tr_{\widetilde{\omega}} \omega  = \frac{1}{\tr_{\widetilde{\omega}} \omega} \left( -K -2 \operatorname{Re}\left( g^{\bar{q} k}\widetilde{T}_{i k}^i \frac{\partial_{\bar{q}} \operatorname{tr}_{\widetilde{g}} g}{\operatorname{tr}_{\widetilde{g}} g}\right) \right.\\ 
        & \left. + g^{\bar{j} i} \widetilde{g}^{\bar{\ell} k}\left(\widetilde{\nabla}_i \overline{\widetilde{T}_{j \ell}^q}-\widetilde{R}_{i \bar{\ell} p j} \widetilde{g}^{\bar{q} p}\right) g_{k \bar{q}} - g^{\bar{j} i} \widetilde{g}^{\bar{l} k} \widetilde{\nabla}_{\bar{\ell}}\widetilde{T}_{i k \bar{j}} - g^{\bar{j} i} \widetilde{g}^{\bar{l} k} \widetilde{\nabla}_i \widetilde{T}_{\bar{j}k \bar{l}} + \widetilde{R}^{\bar{j}i} g_{i \bar{j}} \right).
\end{aligned}
\end{equation}
To progress, we need to estimate the second line in the above equation. Using \eqref{3RSym} and the bound on $\Rm(\widetilde{\omega})$, we have 
$$
g^{\bar{j} i} \widetilde{g}^{\bar{\ell} k}\left(\widetilde{\nabla}_i \overline{\widetilde{T}_{j \ell}^q}-\widetilde{R}_{i \bar{\ell} p j} \widetilde{g}^{\bar{q} p}\right) g_{k \bar{q}}  \leqslant C_0 \tr_{\omega} \widetilde{\omega} \tr_{\widetilde{\omega}} \omega 
$$
and 
$$
- g^{\bar{j} i} \widetilde{g}^{\bar{l} k} \widetilde{\nabla}_{\bar{\ell}}\widetilde{T}_{i k \bar{j}} - g^{\bar{j} i} \widetilde{g}^{\bar{l} k} \widetilde{\nabla}_i \widetilde{T}_{\bar{j}k \bar{l}} + \widetilde{R}^{\bar{j}i} g_{i \bar{j}} \leqslant \tr_{\omega} \widetilde{\omega} + \tr_{\widetilde{\omega}} \omega 
$$
from which the desired estimate follows. 
\end{proof}

\begin{lem}\label{3metricequiv}
    There exists $C>0$ such that 
    \begin{equation}
        C^{-1} \widetilde{\omega} \leqslant \omega \leqslant C \widetilde{\omega}
    \end{equation}
    for all $t \geqslant 0$. 
\end{lem}
\begin{proof}
We first obtain the metric equivalence between $\omega$ and $\widetilde{\omega}$. To obtain a trace bound, we apply a modified argument of Phong-Strum \cite{phong2009dirichlet}. Let 
$$ u^\sharp = u = e^{-t}\varphi_0, $$
which is uniformly bounded due to Propsoition \ref{2prop2}. Modifying \eqref{2heatu} gives us 
\begin{equation}
    \left(\pdv{}{t} - \Delta_{\omega}\right)u^\sharp = \pdv{u^\sharp}{t} + n + tr_{\omega} \widetilde{\omega}. 
\end{equation}
Consider the following quantity 
    \begin{equation}
        Q:= \log \tr_{\widetilde{\omega}} \omega - A u^\sharp + \frac{1}{u^\sharp+K}
    \end{equation}
    where $A>0$ is a constant to be chosen later and $K>0$ is chosen such that $u+K \in [K/2,2K]$. At a maximum point of $Q$, we have 
    \begin{equation}
        \partial_q Q = \frac{\partial_q \tr_{\widetilde{\omega}}\omega}{\tr_{\widetilde{\omega}}\omega} - \left(A + \frac{1}{(u+K)^2} \right) \partial_q u = 0.
    \end{equation}
    At this point, the gradient term in \eqref{3logevol'} can be estimated as follows
    $$
    \begin{aligned}
        \frac{2}{(\tr_{\widetilde{\omega}} \omega)^2} & \left|\text{Re} \left( g^{\bar{q}k} {\widetilde{T}}^i_{ik} \partial_{\bar{q}} \tr_{\widetilde{\omega}} \omega \right)\right| \\
         & \leqslant 2\left(A + \frac{1}{(u+K)^2 } \right) \text{Re} \left( g^{\bar{q} k} \frac{\widetilde{T}_{ik}^i}{\tr_{\widetilde{\omega}} \omega} \partial_{\bar{q}} u \right) \\
        & \leqslant \frac{2}{u+K} \left( \left( \frac{A(u+K)+1}{\tr_{\widetilde{\omega}}\omega} \right) |\widetilde{T}|_g\right) \left( \frac{|\partial u|_g}{(u+K)}\right)  \\
        & \leqslant C A^2 (u+K)^3 \frac{|\widetilde{T}|_{\widetilde{g}} \tr_{\omega} \widetilde{\omega}}{(\tr_{\widetilde{\omega}} \omega)^2} + \frac{|\partial u|_g^2}{(u+K)^3} .
    \end{aligned}
    $$
    The second inequality was obtained using Cauchy-Schwartz with respect to $g$. The third inequality is Holder's inequality and the final inequality is obtained by the assumption on $\widetilde{T}$ and by assuming $(\tr_{\widetilde{\omega}} \omega)^2 \geqslant A^2(u+K)^3$. If this assumption does not hold, we would immediately yield a bound on $\tr_{\widetilde{\omega}} \omega$ by the uniform bound on $u^\sharp$.

    Utilising the above estimate and the evolution of $u$ given by \eqref{2heatu}, along with Proposition \ref{2prop2}, we have the estimate
    $$
    \begin{aligned}
            0 < \left( \pdv{}{t} - \Delta_{\omega} \right)Q & \leqslant (C - A)\tr_{\omega} \widetilde{\omega}  - \frac{|\nabla u|_g^2}{u+K} - 2 \frac{\pdv{u}{t}}{(u+K)^3} \\
            & \leqslant C + (C - A)\tr_{\omega} \widetilde{\omega},
    \end{aligned}
    $$
    at a maximum point of $Q$. Choosing $A>0$ large enough so that $C -A \leqslant-1$, we have 
    \begin{equation}\label{3tr1}
        \tr_{\omega} \widetilde{\omega} \leqslant C,
    \end{equation}
    at a maximum of $Q$, for some $C>0$. Then Proposition \ref{2prop2} implies 
    \begin{equation}\label{3tr2}
        \tr_{\widetilde{\omega}} \omega \leqslant n (\tr_{\omega} \widetilde{\omega})^{n-1} \frac{\text{det}g^{}}{\text{ det}\widetilde{g}} \leqslant C
    \end{equation}
    at a maximum point of $Q$. Unwrapping the definition of $Q$,  each individual term except $ \log \tr_{\omega} \widetilde{\omega}$ is uniformly bounded. Using \eqref{3tr2} then yields
    \begin{equation}\label{3trbdd}
        C^{-1} \widetilde{\omega} \leqslant \omega \leqslant C \widetilde{\omega}
    \end{equation}
    for some $C>0$ independent of time. 
\end{proof}

\section{Curvature Estimates}

Curvature estimates for $\omega(t)$ now follow from a standard maximum principle argument. The arguments here follow from \cite{sherman2013local} and are simpler since we do not need the localisation carried out there. For completeness, we outline the argument in this section. 

As an intermediate step, we will first need a Calabi-type estimate. Let 
\begin{equation}\label{3Sdef}
    S = |\Psi|_g^2 = |\widetilde{\nabla} g|_g^2.
\end{equation}
Then, we have the following estimate. 
\begin{lem}
The quantity $S$ can be bounded from above by 
    \begin{equation}
\left(\frac{\partial}{\partial t}-\Delta\right) S \leq C\left(S^{3 / 2}+1\right)-\frac{1}{2}\left(|\bar{\nabla} \Psi|^2+|\nabla \Psi|^2\right).
\end{equation}
\end{lem}

\begin{proof}

The time derivative is given by 
$$
\frac{\partial}{\partial t} \Psi_{i j}{ }^k=\frac{\partial}{\partial t} \Gamma_{i j}^k - \frac{\partial}{\partial t} \widetilde{\Gamma}_{i j}^k=-\nabla_i R_j{ }^k + \widetilde{\nabla}_i \widetilde{R}_j{ }^k .
$$

Then 
$$
\begin{aligned}
    \frac{\partial}{\partial t} S & =  \frac{\partial}{\partial t}\left(g^{\bar{a} i} g^{\bar{b} j} g_{k \bar{c}} \Psi_{i j}^k \overline{\Psi_{a b}{ }^c}\right) \\
    & =  R^{\bar{a} i} \Psi_{i j}{ }^k \Psi_{\bar{a}}{ }^j{ }_k+R^{\bar{b} j} \Psi_{i j}{ }^k \Psi^i{ }_{\bar{b} k}-R_{k \bar{c}} \Psi_{i j}{ }^k \Psi^{i j \bar{c}} + \Psi_{i j}{ }^k \Psi^{ij}{}_{k}\\
    & -2 \operatorname{Re}\left(\left(\nabla_i R_j{ }^k - \widetilde{\nabla}_i \widetilde{R}_j{ }^k\right) \Psi^{i j}{ }_k\right).
\end{aligned}
$$
Putting this together yields
$$
\begin{aligned}
\left(\frac{\partial}{\partial t}-\Delta\right) S & =  S -|\bar{\nabla} \Psi|^2-|\nabla \Psi|^2+\left(R^{\bar{r} i}{ }_p{ }^p-R_p{ }^{p \bar{r} i}\right) \Psi_{i j}{ }^k \Psi_{\bar{r}}{ }^j{ }_k+\left(R^{\bar{r} j}{ }_p{ }^p-R_p{ }^{p \bar{r} j}\right) \Psi_{i j}{ }^k \Psi^i{ }_{\bar{r} k} \\
& -\left(R_{k \bar{r} p}{ }^p-R_p{ }^p{ }_{k \bar{r}}\right) \Psi_{i j}{ }^k \Psi^{i j \bar{r}}-2 \operatorname{Re}\left[\left(\nabla_i R_j{ }^k{ }_p{ }^p - \widetilde{\nabla}_i \widetilde{R}_j{ }^k { }_p { }^p +\Delta \Psi_{i j}{ }^k\right) \Psi^{i j}{ }_k\right] .
\end{aligned}
$$
The difference in curvature is can be expressed in terms of torsion by \eqref{3RSym}. This leads to 
    \begin{equation}\label{3Stobound}
\begin{aligned}
 \left(\frac{\partial}{\partial t}-\Delta\right) S &  = S-|\bar{\nabla} \Psi|^2-|\nabla \Psi|^2 \\
& +\left(\nabla_r T_{\bar{q}}{ }^{i \bar{q}}+\nabla_{\bar{q}} T^{\bar{q}}{ }_r{ }^i\right) \Psi_{i j}{ }^k \Psi^{r j}{ }_k+\left(\nabla_r T_{\bar{q}}{ }^{j \bar{q}}+\nabla_{\bar{q}} T^{\bar{q}}{ }_r{ }^j\right) \Psi_{i j}{ }^k \Psi^{i r}{ }_k  \\
&  -\left(\nabla_k T_{\bar{q}}{ }^{r \bar{q}}+\nabla_{\bar{q}} T^{\bar{q}}{ }_k{ }^r\right) \Psi_{i j}{ }^k \Psi^{i j}{ }_r \\
& -2 \operatorname{Re}\left[\left(\nabla_i \nabla_j T^{p k}{ }_p+\nabla_i \nabla_{\bar{q}} T^{\bar{q}}{ }_j{ }^k-T_{i p}{ }^r R_r{ }^p{ }_j{ }^k- \widetilde{\nabla}_i \widetilde{R}_j{ }^k { }_p { }^p+g^{\bar{q} p} \nabla_p \widetilde{R}_{i \bar{q} j}{ }^k\right) \Psi^{i j}{ }_k\right].
\end{aligned}
\end{equation}
We want to bound the terms in the second, third and fourth lines above. 

Using \eqref{3TorsionCnst}, we have 
\begin{equation}\label{3NablaTorsion}
    \nabla_{\bar{a}} T_{i j}{ }^k=g^{\bar{j} k}\left(\widetilde{\nabla}_{\bar{a}} \widetilde{T}_{i \bar{l}}-\Psi_{\bar{a} \bar{l}}^{\bar{\imath}} \widetilde{T}_{i j \bar{r}}\right).
\end{equation}
The gradient of torsion is given by a symmetrisation of curvature,  \eqref{3RSym}. From the assumption \eqref{T1Assumption}, we obtain the bound 
\begin{equation}
    |\nabla_{\bar{a}} T_{i j}{ }^k| \leqslant C(S^\frac{3}{2}+1).
\end{equation}
Taking another covariant derivative of \eqref{3NablaTorsion} gives us 
$$
\begin{aligned}
\nabla_a \nabla_b T_{\overline{i j}}{ }^{\bar{k}}= & g^{\bar{k} l}\left(\nabla_a\left(\widetilde{\nabla}_b \widetilde{T}_{\overline{i j}l}-\Psi_{b l}{ }^r \widetilde{T}_{\overline{i j}r}\right)\right) \\
= & g^{\bar{k} l}\left(\widetilde{\nabla}_a \widetilde{\nabla}_b \widetilde{T}_{\overline{i j} l}-\Psi_{a b}{ }^r \widetilde{\nabla}_r \widetilde{T}_{\overline{i j l}}-\Psi_{a l}{ }^r \widetilde{\nabla}_b \widetilde{T}_{\overline{i j} r}\right. \\
& \left.-\left(\nabla_a \Psi_{b l}{ }^r\right) \widetilde{T}_{\overline{i j} r}-\Psi_{b l}{ }^r \widetilde{\nabla}_a \widetilde{T}_{\overline{i j} r}+\Psi_{b l}{ }^r \Psi_{a r}{ }^s \widetilde{T}_{\overline{ij}s}\right),
\end{aligned}
$$
and thus, the two terms of this type in \eqref{3Stobound} can be bounded by 
\begin{equation}
    \leqslant C(S+|\nabla \Psi|+|\overline{\nabla} \Psi|+1).
\end{equation}
The last three terms in \eqref{3Stobound} can be estimated as follows:
\begin{equation}
    |T_{i p}{ }^r R_r{ }^p{ }_j{ }^k| + |\widetilde{\nabla}_i \widetilde{R}_j{ }^k { }_p { }^p|= |g^{\bar{s} r} g^{\bar{q} p} \widetilde{T}_{i p \bar{s}}\left(\widetilde{R}_{r \bar{q} j}{ }^k-\nabla_{\bar{q}} \Psi_{r j}{ }^k\right)| + C \leqslant C(|\bar{\nabla} \Psi|+1)
\end{equation}
and finally 
\begin{equation}
    |\nabla_p \widetilde{R}_{i \bar{q} j}{ }^k| = |\widetilde{\nabla}_p \widetilde{R}_{i \bar{q} j}{ }^k-\Psi_{p i}{ }^r \widetilde{R}_{r \bar{q} j}{ }^k-\Psi_{p j}{ }^r \widetilde{R}_{i \bar{q} r}{ }^k+\Psi_{p r}{ }^k \widetilde{R}_{i \bar{q} j}{ }^r| \leqslant C\left(S^{1 / 2}+1\right).
\end{equation}
Combining all these yields 
\begin{equation}\label{3Sevolest}
\left(\frac{\partial}{\partial t}-\Delta\right) S \leq C\left(S^{3 / 2}+1\right)-\frac{1}{2}\left(|\bar{\nabla} \Psi|^2+|\nabla \Psi|^2\right).
\end{equation}
\end{proof}

\begin{lem}
    Let $S$ be defined as in \eqref{3Sdef}. There exist $C= C(\widetilde{T},\widetilde{R})$ such that 
    \begin{equation}
        S \leqslant C
    \end{equation}
    for all $t \geqslant 0$.
\end{lem}

\begin{proof}
From the evolution of the trace, we have 
$$
\begin{aligned}
& \left(\frac{\partial}{\partial t}-\Delta\right) \operatorname{tr}_{\widetilde{g}} g=-g^{\bar{j} p} g^{\bar{q} i} \widetilde{\nabla}_k g_{i \bar{j}} \widetilde{\nabla}^k g_{p \bar{q}}-2 \operatorname{Re}\left(g^{\bar{j} i} \widetilde{T}_{k i}^p \widetilde{\nabla}^k g_{p \bar{j}}\right) \\
& \quad+g^{\bar{j} i}\left(\widetilde{\nabla}_i \widetilde{T}_{\bar{j}}^{k \bar{q}}-\widetilde{R}_i^{k \bar{q}_{\bar{j}}}\right) g_{k \bar{q}}-g^{\bar{j} i}\left(\widetilde{\nabla}_i \widetilde{T}_{\bar{j} \bar{q}}^{\bar{q}}+\widetilde{\nabla}^k \widetilde{T}_{i k \bar{j}}\right) \\
& +g^{\bar{j} i} \widetilde{T}_{\bar{j}}^{k \bar{q}} \widetilde{T}_{i k}^p(\widetilde{g}-g)_{p \bar{q}} .
\end{aligned}
$$
Combining this with the bounds
\begin{equation}
\left|\nabla \operatorname{tr}_g g\right|^2 \leq C S
\end{equation}
and 
\begin{equation}\label{3nablaS}
    |\nabla S|^2=\left.\left.|\nabla| \Psi\right|^2|| \bar{\nabla}|\Psi|^2|\leq 2| \Psi\right|^2\left(|\nabla \Psi|^2+|\bar{\nabla} \Psi|^2\right) \leqslant 2 S\left(|\bar{\nabla} \Psi|^2+|\nabla \Psi|^2\right)
\end{equation}
we have the estimate 
\begin{equation}\label{3Trevolest}
    \left(\frac{\partial}{\partial t}-\Delta\right) \operatorname{tr}_{\widetilde{g}} g \leq-\frac{S}{C_0}+C\left(S^{1 / 2}+1\right).
\end{equation}
We will apply a maximum principle argument on 
\begin{equation}
Q:= \frac{S}{A-\operatorname{tr}_{\widetilde{g}} g}+B \operatorname{tr}_{\widetilde{g}} g
\end{equation}
where $A>0$ is chosen large enough so that $A - \tr_{\widetilde{\omega}} \omega \in [A/2,2A]$ and $B>0$ to be chosen later. At a maximum point of $Q$, we have 
\begin{equation}\label{3Q2max}
     0 = \nabla Q = \frac{\nabla S}{A- \tr_{\widetilde{\omega}} \omega} + \left( \frac{S}{(A-\tr_{\widetilde{\omega}} \omega)^2} -B \right) \nabla \tr_{\widetilde{\omega}} \omega.
\end{equation}
Combining \eqref{3Sevolest},\eqref{3Trevolest} and \eqref{3Q2max}, we have that at a maximum point of Q,
$$
\begin{aligned}
0 \leqslant \left(\frac{\partial}{\partial t}-\Delta\right) Q \leqslant & \left(-\frac{B}{2 C_0} +\frac{C}{A^2} +C \right)S+\left(-\frac{S}{4 A^2 C_0}+ \frac{C B}{A} \right)S+C B. \\
\end{aligned}
$$
The first bracket is negative for any fixed $B> 2CC_0(1+ A^{-2})$ and the second bracket can be assumed to be negative since otherwise, we would have 
$$
S < 4CC_0AB
$$
which would imply a uniform bound. Thus, we have 
$$\
S \leqslant C
$$
at a maximum of $Q$. Rearranging $Q$ and the uniform bounds on $\tr_{\widetilde{\omega}} \omega$ imply the claim. 
\end{proof}

Finally, we use the Calabi-type estimate to obtain curvature bounds. The evolution of the normalised Chern-Ricci flow can be estimated with the same bound as in the unormalised case in \cite{sherman2013local}. Thus, we have 
$$
\left(\frac{\partial}{\partial t}-\Delta\right)|\mathrm{Rm}|^2 \leqslant C\left(|\mathrm{Rm}|^3+|\mathrm{Rm}|+ |\nabla \Psi|^2+|\bar{\nabla} \Psi|^2-|\nabla \mathrm{Rm}|^2 + 1\right) .
$$
Furthermore, the bound on $S$ allows us to estimate its evolution by 
$$
\left(\frac{\partial}{\partial t}-\Delta\right) S \leqslant C-\frac{1}{2}\left(|\bar{\nabla} \Psi|^2+|\nabla \Psi|^2\right).
$$
Following a similar maximum principle as applied in the Calabi-type estimate, we let 
\begin{equation}
    Q:= \frac{|\mathrm{Rm}|^2}{A-S}+B S 
\end{equation}
for some $A >0$ such that $A-S \in [A/2,A]$ and $B>0$ to be determined later. 
At a maximum point of $Q$, 
$$ \nabla |\Rm|_g^2 = - B(A-S)\nabla S - \frac{|\mathrm{Rm}|^2}{A-S} \nabla S. $$
The parabolic evolution of $Q$ at this maximum point is given by 
$$
0 < \left(\frac{\partial}{\partial t}-\Delta\right) Q =  \frac{1}{A-S}\left(\frac{\partial}{\partial t}-\Delta\right)|\operatorname{Rm}|^2 + \left( \frac{|\operatorname{Rm}|^2}{(A-S)^2} + B \right) \left(\frac{\partial}{\partial t}-\Delta\right) S +\frac{2 B|\nabla S|^2}{A-S} .
$$
The last term above can be estimated by \eqref{3nablaS}. Substituting the estimates for the evolution of $|\Rm|^2_g$ and $S$ from above yields
$$
\begin{aligned}
    0 & < \frac{2C}{A}|\Rm|_g^3 - \frac{2}{A^2}\left(|\bar{\nabla} \Psi|^2 + |\nabla \Psi|^2 \right) |\Rm|_g^2 + \frac{4}{A^2}|\Rm|_g^2 \\
    & - \left(\frac{B}{2}  - \frac{4BC}{A} - \frac{2C}{A} \right)\left(|\bar{\nabla} \Psi|^2 + |\nabla \Psi|^2 \right) +BC.
\end{aligned}
$$
To treat the curvature terms, we note that 
\begin{equation}\label{3quadcurv}
    |\mathrm{Rm}|^2 \leqslant \left(|\bar{\nabla} \Psi|^2 + |\nabla \Psi|^2\right) + C
\end{equation}
and modify our estimate to 
$$
\begin{aligned}
    0 & < \left(\frac{2C}{A}- \frac{2}{A^2}|\Rm|_g \right) \left(|\bar{\nabla} \Psi|^2 + |\nabla \Psi|^2 \right) |\Rm|_g \\
    & - \left( B\left( \frac{1}{2}  - \frac{4C}{A}\right) - \frac{2C}{A} - \frac{4}{A^2}\right)\left(|\bar{\nabla} \Psi|^2 + |\nabla \Psi|^2 \right) +BC.
\end{aligned}
$$
If we fix $A>16C$ and $B > 8C/A + 16/A^2 $, then the second term above will be negative. We assume that the $(2C)/A- (2|\Rm|_g)/A^2 < - 1$, otherwise, $|\Rm|_g \leqslant A^2/2+ AC$ at a maximum point of $Q$. This in turn will lead to a uniform bound on $|\Rm|$. Putting this all together, we are left with 
$$
    \left(|\bar{\nabla} \Psi|^2 + |\nabla \Psi|^2 \right) |\Rm|_g < BC
$$
which again by \eqref{3quadcurv} yields
$$ |\Rm|_g^3 \leqslant C'$$
at the maximum point of $Q$ for some $C'>0$. The desired estimate now follows from the definition of $Q$ and the uniform bound of $S$. The torsion is bounded from \eqref{3TorsionCnst} and \eqref{3trbdd}.

\section{Relation to Classification of Complex Surfaces}

When $X$ is a complex surface, the question regarding independence of singularity type is related to the Kodaria classification of surfaces. The following behavior is conjectured by Tosatti and Weinkove in \cite{tosatti2021chern}. 

\begin{thm}\label{classification}
    Let $X$ be a compact complex surface with $K_X$ nef, and $\omega(t)$ a solution of the Chern-Ricci flow on $X$ starting at a Gauduchon metric $\omega_0$. Then the flow develops a type III singularity at infinity, 
    \begin{enumerate}
        \item in the case where $X$ is a minimal surface of general type,  if and only if $K_M$ is ample 
        \item (conjectured) in the case where $X$ is a minimal properly elliptic surface, if and only if the only singular fibers are multiples of a smooth fiber (i.e. of Kodaira type $m I_0, m \geqslant$ 2)
        \item in the case where $X$ is a K\"ahler Calabi-Yau surface, if and only if $M$ is finitely covered by a torus and $\omega_0$ is K\"ahler
        \item in the case where $X$ is Kodaira surface, never,
        \item (conjectured) in the case where $X$ is an Inoue surface, always.
    \end{enumerate}
\end{thm}

The theorem above is complete except in cases (2), for minimal properly elliptic surfaces, and (5), for Inoue surfaces. In these cases, progress has been made in \cite{angella2023leafwise} where the initial metric is in a $\partial \bar{\partial}$ class of an explicit metric constructed by Tricerri and Vaisman. The Tricerri and Vaisman metrics are constructed from the universal cover of $X$ and provide reference metrics from which analysis similar to that in Gross-Tosatti-Zhang \cite{gross1013collapsing} can be carried out. 

Theorem \ref{T1} can be seen as a step towards completing Theorem \ref{classification} especially due to evidence in \cite{angella2023leafwise}.

\bibliographystyle{plain}
\bibliography{ref}

\end{document}